\documentclass[12pt, reqno]{amsart}
\usepackage{amsmath, amsthm, amscd, amsfonts, amssymb, graphicx, color}
\usepackage[bookmarksnumbered, colorlinks, plainpages]{hyperref}
\hypersetup{colorlinks=true,linkcolor=blue, anchorcolor=blue, citecolor=blue, urlcolor=blue, filecolor=blue, pdftoolbar=true}

\newcommand{\R}{{\mathbb{R}}}

\newcommand{\N}{{\mathbb{N}}}

\newcommand{\F}{{\mathcal{F}(X)}}

\newcommand{\eps}{\varepsilon}

\newcommand{\loglike}[1]{\mathop{\rm #1}\nolimits}

\newcommand{\Lip}{\loglike{Lip}_0}
\newcommand{\SA}{\loglike{SA}}
\newcommand{\DA}{\loglike{DA}}
\newcommand{\LDA}{\loglike{LDA}}

\newcommand{\dist}{\mathrm{dist}}
\newcommand{\diam}{\mathrm{diam}}
\newcommand{\1}{{\bf 1}}
\newcommand{\conv}{{\mathrm{conv}}}
\newcommand{\cconv}{\overline{\mathrm{conv}}}

\newcommand{\LipBPB}{ \loglike{LipBPB}}

\newcommand{\LLipBPB}{ \loglike{LLipBPB}}
\newcommand{\bea}{\begin{eqnarray*}}
\newcommand{\eea}{\end{eqnarray*}}
\newcommand{\beq}{\begin{equation}}
\newcommand{\eeq}{\end{equation}}

\newcommand{\supp}{{\mathrm{supp}}}
\newcommand{\sign}{{\mathrm{sign}}}

\numberwithin{equation}{section}

\theoremstyle{plain}
\newtheorem{theorem}{Theorem}[section]
\newtheorem{corollary}[theorem]{Corollary}
\newtheorem{prop}[theorem]{Proposition}
\newtheorem{lemma}[theorem]{Lemma}

\theoremstyle{definition}

\newtheorem{remark}[theorem]{Remark}
\newtheorem{example}[theorem]{Example}

\newtheorem{definition}[theorem]{Definition}

\renewcommand{\leq}{\leqslant}
\renewcommand{\geq}{\geqslant}
\renewcommand{\le}{\leqslant}
\renewcommand{\ge}{\geqslant}

\begin{document}
\setcounter{page}{1}

\title{Norm attaining Lipschitz functionals}
\author[V.~Kadets, M.~Mart\'{\i}n, M.~Soloviova]{Vladimir Kadets$^1$, Miguel Mart\'{\i}n$^2$, and Mariia Soloviova$^1$}

\address{$^1$ Department of Mathematics and Informatics, Kharkiv V.~N.~Karazin National University, pl.~Svobody~4,
61022~Kharkiv, Ukraine}
\email{\textcolor[rgb]{0.00,0.00,0.84}{vova1kadets@yahoo.com, mariiasoloviova93@gmail.com}}

\address{$^2$ Departamento de An\'{a}lisis Matem\'{a}tico \\ Facultad de
 Ciencias \\ Universidad de Granada \\ 18071 Granada\\ Spain }
\email{\textcolor[rgb]{0.00,0.00,0.84}{mmartins@ugr.es}}



\date{January 28th, 2016}

\begin{abstract}
We prove that for a given Banach space $X$, the subset of norm attaining  Lipschitz functionals in $\Lip(X)$ is weakly dense but not strongly dense. Then we introduce a weaker concept of directional norm attainment and demonstrate that for a uniformly convex  $X$ the set of directionally  norm attaining  Lipschitz functionals is strongly dense in $\Lip(X)$ and, moreover, that an analogue of the Bishop-Phelps-Bollob\'{a}s theorem is valid.
\end{abstract}

\maketitle

\section{Introduction and motivation}
In this text, the letter $X$ stands for a real Banach space. We denote, as usual, by $S_X$ and $B_X$  the unit sphere and the closed unit ball of $X$, respectively.  A functional $x^* \in X^*$ \emph{attains its norm}, if there is $x \in S_X$ with $x^*(x) = \|x^*\|$. If $X$ is reflexive, then all  $x^* \in X^*$ attain their norms and, according to the famous James theorem (see \cite[Chapter 1, theorem 3]{Diestel}), in every non-reflexive space there are functionals that do not attain their norm. Nevertheless, in every Banach space there are ``many'' norm attaining functional. Namely, the classical Bishop-Phelps theorem  (\cite{Bishop-Phelps}, see also \cite[Chapter 1]{Diestel})  states that the set of norm attaining functionals on a Banach space is norm dense in the dual space. Moreover, for every closed bounded convex set $C \subset X$, the collection of functionals that attain their maximum on $C$ is  norm dense in $X^*$.

The fact that every functional can be approximated by norm attaining ones is quite useful, but sometimes one needs more. Namely, sometimes (in particular, when one works with numerical radius of operators) one needs to approximate a pair ``element and functional'' by a pair  $(x, x^*)$ such that $x^*$ attains its norm in $x$. Such a modification of the Bishop-Phelps theorem was given by  B.~Bollob\'{a}s
\cite{Bollobas}. Below we cite it in a slightly modified form with sharp estimates  \cite{C-K-M-M-R}.

\begin{theorem}[Bishop-Phelps-Bollob\'{a}s theorem]\label{th:BPB}
Let $X$ be a Banach space. Suppose $x\in B_X$ and $x^*\in B_{X^*}$
satisfy $ x^*(x) \ge 1 -  \delta$ for $\delta \in (0, 2)$. Then there exists $(y,y^*)\in X\times X^*$ with $\|y\|=\|y^*\|=y^*(y)=1$ such that
\begin{equation} \label{ineq-sharp}
 \max\{\|x-y\|, \|x^*-y^*\|\} \leq \sqrt{2\delta}.
 \end{equation}
\end{theorem}

In this project we are searching for possible extensions of the Bishop-Phelps theorem and the Bishop-Phelps-Bollob\'{a}s theorem for non-linear  Lipschitz functionals $f:X \longrightarrow \R$.

In the sequel we use the letter $E$ to denote a metric space, equipped with a distinguished point $0$ and  such that $E \setminus \{0\} \neq \emptyset$, and $\rho$ denotes the given distance of $E$. Recall that to such a \emph{pointed metric space} one  associates the Banach space  $\Lip(E)$ that consists of functions $f: E \longrightarrow \R$ with $f(0)=0$ which satisfy (globally) the  Lipschitz condition. This space is equipped with the norm
\beq \label{lipnorm}
\|f\|={\sup}\left\{\frac{\vert f(x)-f(y)\vert}{\rho(x,y)}: x,y \in E, x\neq y\right\}.
\eeq
In other words, $\|f\|$ is the smallest Lipschitz constant of $f$. We refer the reader to the book \cite{Weaver} for background on Lipschitz spaces.

The most interesting results of our paper are related to definition \ref{dir-att def} below, which does not make sense for functions on general pointed metric spaces, but only in the Banach space setting. That is why in this paper we mainly concentrate on the case of  $\Lip (X)$, where $X$ is a Banach space. Remark that in this case, evidently, $X^*$ is a closed subspace of $\Lip (X)$ with equality of norms.

As this paper deals with possible extensions of the Bishop-Phelps and Bishop-Phelps-Bollob\'{a}s theorems, we first have to say what we understand by a norm attaining Lipschitz functional. We have a couple of possible definitions for this.

First, the most natural definition of norm attainment for a functional $f \in \Lip (E)$ is the following.

\begin{definition}\label{str-att def}
A functional $f \in \Lip (E)$ \emph{attains its norm in the strong sense} if there are $ x,y \in E$, $x\neq y$ such that $\|f\|=\frac{\vert f(x)-f(y)\vert}{\rho(x,y)}$. The subset of all  functionals $f \in \Lip (E)$ that attain their norm in the strong sense is denoted $\SA(E)$.
\end{definition}

Unfortunately, in the sense of the Bishop-Phelps theorem, this definition is too restrictive. Even in the one-dimensional case ($X = \R$), the subset $\SA(X)$ is not dense in $\Lip (X)$ (see Example~\ref{str-att thm} and Theorem \ref{m_conv_theor}). Nevertheless, for every Banach space $X$, $\SA(X)$ is weakly sequentially dense in $\Lip (X)$ (see Theorem~\ref{c_0 theor}). This is the content of our section~\ref{sec:SA}, where the results are actually proved for metrically convex metric spaces.

It is then clear that a less restrictive way for a Lipschitz functional to attain its norm is needed to get density. We will use the following definition.

\begin{definition}\label{dir-att def}
A functional $g \in \Lip (X)$ \emph{attains its norm at the direction} $u\in S_X$ if there is a sequence of pairs $\{(x_n,y_n)\}$ in $X\times X$, with $x_n\neq y_n$, such that
\begin{equation*} 
\lim_{n \to \infty}\frac{x_n-y_n}{\|x_n-y_n\|} = u \quad \text{and} \quad  \lim_{n \to \infty}\frac{ g(x_n)-g(y_n)}{\|x_n-y_n\|} = \|g\|.
\end{equation*}
In this case, we say that $g$ \emph{attains its norm directionally}. The set of all those $f \in \Lip (X)$ that attain their norm directionally is denoted by $\DA(X)$.
\end{definition}

We start our consideration with two reasons of why the directional approach is natural in our framework.

\begin{enumerate}
\item[(a)] If $X$ is finite-dimensional, then $\DA(X) = \Lip (X)$ by a compactness argument, so at least in this easiest case the directional Bishop-Phelps theorem does not fail.
\item[(b)] A linear functional attains its norm at direction $u$ if and only if $ f(u) = \|f\|$, so it attains its norm in the usual sense.
\end{enumerate}

We devote section~\ref{sec:seminorms} to study norm attaining seminorms. Continuous seminorms on a Banach space $X$ form a closed cone in $\Lip(X)$, and the Lipschitz norm coincides with the uniform norm (i.e.\ the supremum on the unit sphere of the space). Moreover, the respective sets of continuous seminorms, that attain the norm strongly, that attain the norm directionally, and that attain the norm uniformly, coincide (Lemma~\ref{lemma:NAseminorm}). We provide a general Bishop-Phelps-Bollob\'{a}s theorem for seminorms, but getting uniform density instead of density in the Lipschitz norm (Proposition \ref{prop:seminDensistyUnifNorm}). Besides, we prove the Lipschitz-norm density of norm attaining seminorms for Banach spaces with the Radon-Nikod\'{y}m property (Proposition \ref{prop:semiRNP}), by using the Stegall's version of  the classical Bourgain-Steagall non-linear optimization principle. Finally, we show that in every infinite-dimensional Banach space, there is a continuous seminorm which does not attain its norm (Example~\ref{example:NAseminorm}), so item (a) above actually characterizes finite dimension.

The main result of the paper is a Bishop-Phelps-Bollob\'{a}s-type theorem for Lipschitz functionals on uniformly convex spaces. Even though we only have one kind of examples, it makes sense to introduce the following definition.

\begin{definition} \label{dir-LipBPB def}
A Banach space $X$ has the \emph{directional Bishop-Phelps-Bollob\'{a}s property for Lipschitz functionals} ($X \in  \LipBPB$ for short), if for every $\varepsilon>0$ there is such a $\delta>0$, that for every $f\in \Lip (X)$ with  $\|f\| = 1$ and for every $ x,y \in X$ with $x\neq y$ satisfying $ \frac{ f(x)-f(y)}{\|x-y\|}>1-\delta$, there is $g \in \Lip (X)$ with $\|g\|=1$ and there is $u \in S_X$ such that $g$ attains its norm at the direction $u$,
$\|g-f\|<\varepsilon$, and  $\left\|\frac{x-y}{\|x-y\|}-u\right\|<\varepsilon$.
\end{definition}

So, with this notation, the main result of the paper is to prove that uniformly convex spaces have the $\LipBPB$ (Theorem~\ref{main theorem}), and even a stronger property called \emph{local directional Bishop-Phelps-Bollob\'{a}s property for Lipschitz functionals}  introduced in Definition~\ref{loc-dir-LipBPB def}. This is the content of section~\ref{sec:main}. In the way to prove such result, we need to provide a weak version of the property (Lemma~\ref{lemLipBPB-prel}), valid for all Banach spaces, which is proved using the Lipschitz-free space (see definition in section~\ref{sec:freeLipspace}). We also prove (Lemma~\ref{lemma-relaxed}) that the requirements for a general Banach space to have the (local) Bishop-Phelps-Bollob\'{a}s property for Lipschitz functionals can be relaxed in the sense that only ``approximate'' directional norm attainment of $g$ is required. These two preliminary results are the content of section~\ref{sec:freeLipspace}.

We should make it clear that we are not able to answer some natural easy-looking questions related to our results. For example, we are not able to construct a Banach space which has no (local) directional Bishop-Phelps-Bollob\'{a}s property for Lipschitz functionals.
On the other hand, we refer to \cite{GodeLips} for some negative results on norm attaining Lipschitz maps between Banach spaces which do not overlap with the results of this manuscript.

We finish this introduction recalling an important tool to construct Lipschitz functionals: the classical \emph{McShane's extension theorem}. It says that if $M$ is a subspace of a metric space $E$ and $f:  M\longrightarrow \R$ is a Lipschitz functional, then there is an extension to a Lipschitz functional
$F:  E \longrightarrow \R$ with the same Lipschitz constant; see
\cite[Theorem~1.5.6]{Weaver} or \cite[p.~12/13]{BL1}.


\section{Strongly attaining Lipschitz functionals}\label{sec:SA}

As announced in the introduction, there is no Bishop-Phelps type theorem for Lipschitz functionals in the strong sense of the attainment, even in the one-dimensional case.

\begin{example} \label{str-att thm}
{\slshape $\SA([0, 1])$ is not dense in $\Lip([0, 1])$.}
\end{example}

In order to demonstrate this, we need the following easy lemma.

\begin{lemma} \label{str-att lemma}
If $f \in \Lip (E)$ attains its norm on a pair $(x,y) \in E\times E$, $x\neq y$, and $z \in E\setminus\{x,y\}$ is such an element that $\rho(x,y) = \rho(x,z) + \rho(z,y)$, then $f$ strongly attains its norm on the pairs $(x,z)$ and $(y,z)$, and
\beq\label{str-att lemma-eq}
f(z) = \frac{\rho(z,y)f(x) +\rho(x,z)f(y)}{\rho(x,y)}.
\eeq
In particular, if $E$ is a convex subset of a Banach space, then $f$ is affine on the closed segment $\conv\{x, y\}$, i.e.  $f(\theta x + (1 - \theta)y) = \theta f(x) +  (1 - \theta) f(y)$ for every $\theta \in [0, 1]$.
\end{lemma}

\begin{proof}
We may (and do) assume without loss of generality that $f(x) - f(y) \ge 0$ (otherwise we multiply $f$ by $-1$). Since $f$ attains its norm on the pair $(x,y)$, we have
\bea
\|f\| \rho(x,y) &=& f(x) - f(y) = f(x)  - f(z) + f(z) -  f(y)  \\
&\le&  \|f\| \rho(x,z)  + \|f\| \rho(z,y)  = \|f\| \rho(x,y).
\eea
This means that the inequalities
$$
f(x)  - f(z) \le \|f\|  \rho(x,z)\ \text{ and }\
 f(z) -  f(y) \le \|f\| \rho(z,y)
$$
which we used above are, in fact, equalities. So, $f(x) = f(z) + \|f\|  \rho(x,z)$ and $f(y) = f(z) - \|f\| \rho(z,y)$. Substituting the last two formulas into the right-hand side of \eqref{str-att lemma-eq}, we get the desired result.
\end{proof}

\begin{proof}[Proof of Example~\ref{str-att thm}]
It is well known (for example, \cite[Example 1.6.5]{Weaver} or \cite[Propositions 6 and 2]{Kadets}) that $\Lip([0, 1])$ is isometric to $L_\infty([0, 1])$, and the corresponding bijective isometry $U: \Lip([0, 1]) \longrightarrow L_\infty([0, 1])$ is just the differentiation operator (the derivative of a Lipschitz function $f: [0, 1] \longrightarrow \R$ exists almost everywhere). Under this isometry, every $f \in \SA ([0, 1])$ maps to a function, which is equal either to $ \|f\|$ or to $- \|f\|$ on some non-void interval (here we use Lemma \ref{str-att lemma}). Denote $A$ a nowhere dense closed subset of $[0,1]$ of positive Lebesgue measure. and let  $g \in \Lip([0, 1])$ be the function, whose derivative equals $\1_A$ (the characteristic function of $A$) a.e. Then $g$  cannot be approximated by functions from $\SA([0, 1])$.
Actually, we claim that
\beq \label{norm-dif-eq}
\|g - f\| = \| \1_A - f'\|_\infty \ge \frac12
\eeq
for every $f \in \SA (\R)$.

 In fact, $\|g\| = \|\1_A\|_\infty = 1$ so, if $\|f'\|_\infty \le \frac12$, then \eqref{norm-dif-eq} follows from the triangle inequality. If $\|f'\|_\infty > \frac12$, then, as we remarked before, $|f'(t)| > \frac12$ on some open interval $(a, b)$. But, since  $A$ is nowhere dense, there is a smaller interval $(c, d) \subset (a, b)$ such that $\1_A(t) = 0$ for $t \in (c, d)$. So $| (\1_A - U(f))(t)| > \frac12$  on $(c, d)$, which implies \eqref{norm-dif-eq}.  Hence, $\SA ([0, 1])$ is not dense in $\Lip([0, 1])$.
\end{proof}

Recall, that a metric space $E$ is said to be \emph{metrically convex} if, for every pair of distinct points $x, y \in E$, there is a curve $\ell \subset E$ which connects $x$ and $y$ and is isometric to the segment $[0, \rho(x,y)] \subset \R$. The next theorem demonstrates that  Example~\ref{str-att thm} extends to all non-trivial metrically convex pointed metric spaces, and in particular to all Banach spaces.

\begin{theorem} \label{m_conv_theor}
Let $E$ be a metrically convex pointed metric space. Then  $\SA (E)$ is not dense in $\Lip(E)$.
\end{theorem}

\begin{proof}
Fix $x_0 \in E \setminus \{0\}$. Without loss of generality we may assume $\rho(0, x_0) = 1$.  Denote $\ell \subset E$ the isometric copy of $[0, 1]$ which connects 0 and $x_0$. Again, without loss of generality we may assume that $\ell = [0, 1] \subset E$. Let $u: E \longrightarrow [0, 1]$ be a norm-1 Lipschitz function whose restriction to $[0, 1]$ is the identity map (here we apply McShane's extension theorem) and let $g \in \Lip([0, 1])$ be the function from Example~\ref{str-att thm}. We are going to demonstrate that $h := g \circ u \in \Lip(E) \setminus \overline{\SA(E)}$.

Consider arbitrary $f \in \SA (E)$. Like in  Example~\ref{str-att thm}, we will show that\linebreak $\|h - f\|  \ge \frac12$. Assume contrary
\beq \label{norm-dif-eq+}
\|h - f\|  < \frac12.
\eeq
Then $\|f\| > 1/2$. Denote  $(x,y) \in E\times E$, $x\neq y$ a pair at which  $f$ attains its norm. Due to  metric convexity of $E$, there is an isometric copy $\gamma$ of $[0, \rho(x,y)]$ which connects $x$ and $y$. According to Lemma \ref{str-att lemma},
$$
|f(z_1) - f(z_2)| = \|f\| \rho(z_1, z_2)
$$
for every $z_1, z_2 \in \gamma$, $z_1 \neq z_2$. Consequently,
$$
|h(z_1) - h(z_2)| \ge |f(z_1) - f(z_2)|  - \|h - f\|\rho(z_1, z_2) > 0
$$
 for all such $z_1, z_2$. This, in turn, implies that $g(v_1) \neq g(v_2) $ for every pair of distinct points $v_1, v_2$ in the non-void interval $u(\gamma) \subset [0,1]$. But, according to the definition of $g$, this is impossible.
\end{proof}

Even though in the sense of Definition \ref{str-att def} the Bishop-Phelps theorem does not transfer to Lipschitz functionals, $\SA(X)$ cannot be too small. Namely, we are going to demonstrate that $\SA(X)$ is weakly sequentially dense in $\Lip(X)$ for every Banach space $X$.  Even more, we will prove an analogous fact for a wider class of  ``local'' metric spaces.

According to \cite[Definition 2.2]{IvKadWer}, a pointed metric space $E$ is said to be \emph{local}  if for every $\eps > 0$ and for
every function $f \in \Lip(E)$ there are two distinct points
 $t_1, t_2\in E$ such that $\rho(t_1, t_2) < \eps$ and
\begin{equation} \label{eq6}
\frac{f(t_2) - f(t_1)}{\rho(t_1, t_2)} > \|f\| - \eps.
\end{equation}
Every local space evidently is infinite, and it is also easy to see that locality implies the absence of  isolated points (the function $f = \1_{\{\tau\}}$ where $\tau$ is an isolated point does not fit to the definition). Every metrically convex $E$ is local
\cite[Proposition 2.3]{IvKadWer} and a partial converse statement is known \cite[Proposition 2.9]{IvKadWer}: let $E$ be a metric subspace of  a smooth locally uniformly rotund Banach space. If $E$ is compact and local, then $E$ is convex.

At first, a helpful lemma.

\begin{lemma} \label{c_0 lemma}
Let $E$ be a local metric space, $\{f_n\}$ a sequence in $S_{\Lip (E)}$, and for each $n\in \N$ let  $U_n:=\{x \in E\,:\, f_n(x) \neq 0\}$ be the corresponding supports. Suppose that the sets $U_n$ are pairwise separated, i.e.
$$
d_{n,m} = \inf\{\rho(x,y): x \in U_n, y \in U_m\} > 0
$$
for every $n \neq m$. Then, the sequence $\{f_n\}$ is isometrically equivalent to the canonical basis of $c_0$, i.e.\ for any finite collection $\{a_j\}_{j=1}^n$ of reals
\beq\label{c_0 lemma-eq}
\left\|\sum_{j=1}^n a_j f_j \right\| = \max_{k}|a_k|.
\eeq
\end{lemma}

\begin{proof}
Denote $f = \sum_{j=1}^n a_j f_j$, $d_n = \min\{d_{k,j}: k,j \in \{1,2,\ldots,n\}\}$. Fix an $\eps > 0$ satisfying $\eps < d_n$.  According to the definition of locality,   there are points
 $t_1, t_2\in E$ such that $0 < \rho(t_1, t_2) < \eps$ which fulfill \eqref{eq6}. If one of $t_i$ belongs to some $U_m$, then the other one either belongs to the same $U_m$, or lies outside of $f$'s support. Consequently, for this $m$
 $$
 \|f\| < \frac{f(t_2) - f(t_1)}{\rho(t_1, t_2)} + \eps =   a_m\frac{f_m(t_2) - f_m(t_1)}{\rho(t_1, t_2)} + \eps \le |a_m| + \eps.
 $$
By the arbitrariness of $\eps$,  this implies $ \|f\|  \le \max_{k}|a_k|$.

In order to get the reverse inequality, we fix such an $m$  that $\max_{k}|a_k|  = |a_m|$ and apply the locality condition to $f_m$. We get  points
$t_1, t_2\in E$ such that $0 < \rho(t_1, t_2) < \eps$ and
$$
\frac{f_m(t_2) - f_m(t_1)}{\rho(t_1, t_2)} > 1 - \eps.
$$
If $\eps$ is small enough, this condition again means that  one of $t_i$ belongs to  $U_m$ and the other one either belongs to $U_m$, or lies outside of $f$'s support.
Consequently,
\begin{align*}
\max_{k}|a_k| & = |a_m| \le \frac{1}{1-\eps}  |a_m|  \frac{|f_m(t_2) - f_m(t_1)|}{\rho(t_1, t_2)} \\ & = \frac{1}{1-\eps} \frac{|f(t_2) - f(t_1)|}{\rho(t_1, t_2)} \le  \frac{1}{1-\eps} \|f\|.\qedhere
\end{align*}
\end{proof}

As an obvious consequence, we obtain the following.

\begin{corollary} \label{c_0 coroll}
If $E$ is a local metric space, then every bounded separately supported sequence in $\Lip (E)$ converges weakly to zero.
\end{corollary}

We are now ready to prove the weak sequential density of strongly attaining Lipschitz functionals.

\begin{theorem} \label{c_0 theor}
If $E$ is a local metric space, then $\SA(E)$ is weakly sequentially dense in $\Lip(E)$, i.e.\ for every $g \in \Lip (E)$ there is a sequence $\{g_n\}$ in $\SA(E)$ which converges weakly to $g$.
\end{theorem}

\begin{proof}
Without loss of generality we may assume $\|g\| = 1$. Let us choose a sequence of pairwise disjoint balls $\{U_n\}$ with corresponding radii $r_n > 0$, centers $x_n$ and such that $0 \not\in U_n$. For every $n\in \N$, select $\eps_n \in (0, \frac12)$ and $y_n \in U_n$ with $0 < \rho(x_n,y_n) = \eps_n r_n$ (here we use the absence of isolated points). For a fixed $n$, consider $E_n = (E \setminus U_n) \cup \{x_n, y_n\} \subset E$. Define $h_n: E_n \longrightarrow \R$ as follows: $h_n(t) = g(t)$ for $t \in E_n  \setminus \{x_n\}$ and $h_n(x_n) = g(y_n) - s_n (1+2\eps_n) \rho(x_n,y_n) $, where $s_n = \sign(g(y_n) - g(x_n))$. We claim that the Lipschitz constant of $h_n$ is attained at the pair $(x_n, y_n)$ and equals $1+2\eps_n$. At first,
$$
\frac{\vert h_n(x_n)-h_n(y_n)\vert}{\rho(x_n,y_n)} = 1+2\eps_n.
$$
At second, if $x,y \in E_n  \setminus \{x_n\}$, then
$$
\frac{\vert h_n(x)-h_n(y)\vert}{\rho(x_n,y_n)} = \frac{\vert g(x)-g(y)\vert}{\rho(x,y)} \le 1.
$$
So, it remains to check that for every $y \in E  \setminus U_n$
$$
\frac{\vert h_n(x_n)-h_n(y)\vert}{\rho(x_n,y)}  \le 1+2\eps_n.
$$
 In fact,
\bea
\frac{\vert h_n(x_n)-h_n(y)\vert}{\rho(x_n,y)} &=& \frac{\vert g(y_n)  - s_n (1+2\eps_n) \rho(x_n,y_n)  - g(y)\vert}{\rho(x_n,y)}  \\
&=&  \frac{\vert g(y_n)  - s_n(1+2\eps_n) \eps_n r_n  - g(y)\vert}{\rho(x_n,y)} \\  &\le&  \frac{\vert g(x_n)   - g(y)\vert}{\rho(x_n,y)} + \frac{\vert g(y_n) - g(x_n)  - s_n(1+2\eps_n) \eps_n r_n  \vert}{\rho(x_n,y)} \\  &\le& 1 + \frac{ (1+2\eps_n) \eps_n r_n }{r_n} \le  1+2\eps_n.
\eea
The claim is proved. Now, applying McShane's
extension theorem, we extend $h_n$ to a functional $g_n$ on the whole of $E$ preserving its Lipschitz constant. Then, $g_n \in \Lip (X)$, $\|g_n\| = 1+2\eps_n$ and
$$
\frac{\vert g_n(x_n)-g_n(y_n)\vert}{\rho(x_n,y_n)}=\frac{\vert h_n(x_n)-h_n(y_n)\vert}{\rho(x_n,y_n)} = 1+2\eps_n,
$$
so $g_n \in \SA(E)$. On the other hand, $\supp(g_n - g) \subset U_n$, so the functionals $g_n - g$ are disjointly supported. According to Corollary~\ref{c_0 coroll}, this implies that $\{g_n - g\}$ converges weakly to $0$.
\end{proof}

\begin{remark}\label{c_00 rem}
Choosing in the above proof sequences $\{\eps_n\}$ and $\{r_n\}$ converging to zero, one gets additional properties of the approximating sequence $g_n$. Namely, {\slshape one can get that $\{\|g_n\|\} \longrightarrow \|g\|$ and that $\{g_n\}$ converges uniformly on the whole of $E$ to $g$.}
\end{remark}

As Banach spaces are local metric spaces, we may particularize Theorem~\ref{c_0 theor} and Remark~\ref{c_00 rem} to this case.

\begin{corollary}
Let $X$ be a Banach space. Then, for every $g \in \Lip (X)$ there is a sequence $\{g_n\}$ in $\SA(X)$ which converges weakly to $g$. Moreover, the sequence $\{g_n\}$ can be chosen in such a way that one also has that $\{\|g_n\|\} \longrightarrow \|g\|$ and that $\{g_n\}$ converges uniformly on the whole of $X$ to $g$.
\end{corollary}


\section{Bishop-Phelps theorems for seminorms}\label{sec:seminorms}

Let $X$ be a Banach space. We write $\mathrm{Sem}(X)$ for the set of all continuous seminorms on $X$. As usual, we consider the set $\mathrm{Sem}(X)$ as a closed cone of $\ell_\infty(S_X,\R)$, where for a set $\Gamma$, $\ell_\infty(\Gamma,\R)$ denotes the Banach space of all bounded functions from $\Gamma$ into $\R$ endowed with the uniform norm $\|\cdot\|_\infty$. We need a couple of easy remarks on continuous seminorms: that they can be viewed as Lipschitz functionals with equality of norms and that can be expressed in terms of the norm of a bounded linear operator. We use the notation $L(X,Y)$ for the Banach space of all bounded linear operators between the Banach spaces $X$ and $Y$. Only in this section, we will write $\|\cdot\|_{\textrm{Lip}}$ to denote the Lipschitz norm when there is a possible confusion.

\begin{remark}\label{seminorms-remark} Let $X$ be a Banach space.
\begin{enumerate}
\item[(a)] $\mathrm{Sem}(X)\subseteq \Lip(X)$ and $\|p\|_{\textrm{Lip}}=\|p\|_\infty$ for every $p\in \mathrm{Sem}(X)$.
\item[(b)] For every $p\in\mathrm{Sem}(X)$, there exist a Banach space $Y$ and $T\in L(X,Y)$ such that $p(x)=\|Tx\|$ for all $x\in X$, which obviously satisfies $\|T\|=\|p\|$. Actually, one can consider $Y=\ell_\infty(\Gamma,\R)$ where $\Gamma$ is a set whose cardinality equals the density character of $X$.
\end{enumerate}
\end{remark}

\begin{proof}
(a) is a consequence of the triangle inequality. Indeed, for $p\in\mathrm{Sem}(X)$ and $x,y\in X$ with $x\neq y$, we have
$$
\frac{|p(x)-p(y)|}{\|x-y\|} \leq \frac{p(x-y)}{\|x-y\|}= p\left(\frac{x-y}{\|x-y\|}\right)\leq \|p\|_\infty,
$$
so $\|p\|_{\textrm{Lip}}\leq \|p\|_\infty$. Conversely, for every $z\in S_X$ we have
$$
p(z)=\frac{p(z)-p(0)}{\|z-0\|}\leq \|p\|_{\textrm{Lip}}.
$$
(b). Let $Y$ be the completion of the quotient $(X,p)/\ker p$ and let $T\in L(X,Y)$ be the composition of the natural quotient map from $X$ onto $(X,p)/\ker p$ and the inclusion into $Y$. It is then obvious that $p(x)=\|Tx\|$ for every $x\in X$. For the moreover part, just observe that the density character of $Y$ is smaller or equal than the one of $X$, so $Y$ embeds isometrically into $\ell_\infty(\Gamma,\R)$ and one can view $T$ as an element of $L(X,\ell_\infty(\Gamma,\R))$.
\end{proof}

Our goal in this section is to study norm attaining seminorms. As we have two norms defined on $\mathrm{Sem}(X)$ and also we have two ways for a Lipschitz functional to attain the norm, we have three possibilities. As a matter of fact, all of them are the same for continuous seminorms.

\begin{lemma}\label{lemma:NAseminorm}
Let $X$ be a Banach space and $p\in\mathrm{Sem}(X)$. Then, the following conditions are equivalent:
\begin{enumerate}
\item[(i)] $p$ attains its norm as an element of $\ell_\infty(S_X,\R)$ (i.e.\ there exists $z\in S_X$ such that $p(z)=\|p\|$),
\item[(ii)] $p\in \SA(X)$ (i.e.\ there exists a pair $(x,y)\in X\times X$ with $x\neq y$ such that $\frac{|p(x)-p(y)|}{\|x-y\|}=\|p\|$),
\item[(iii)] $p\in \DA(X)$ (i.e.\ there exists a sequence of pairs $\{(x_n,y_n)\}$ in $X\times X$ with $x_n\neq y_n$ for all $n$ such that $\left\{\frac{x_n-y_n}{\|x_n-y_n\|}\right\}$ is convergent and $\left\{\frac{p(x_n)-p(y_n)}{\|x_n-y_n\|}\right\}\longrightarrow \|p\|$,
\item[(iv)] for every Banach space $Y$ and every such an operator $T\in L(X,Y)$ that $p(x)=\|Tx\|$ ($x\in X$), $T$ attains its norm (i.e.\ there is $z\in S_X$ such that $\|T z\|=\|T\|=\|p\|$),
\item[(v)] there exist a Banach space $Y$ and a norm attaining operator $T\in L(X,Y)$ such that $p(x)=\|Tx\|$ ($x\in X$).
\end{enumerate}
In case that $p$ satisfies any (all) of the above condition, we will say that $p$ is a \emph{norm attaining seminorm}.
\end{lemma}

\begin{proof}
For (i)$\Rightarrow$(ii), consider the pair $(z,0)$. (ii)$\Rightarrow$(iii), (iv)$\Rightarrow$(v) and (v)$\Rightarrow$(i) are evident.

(iii)$\Rightarrow$(iv). Write $u=\lim \frac{x_n-y_n}{\|x_n-y_n\|}$, so $T(u)=\lim \frac{T(x_n)-T(y_n)}{\|x_n-y_n\|}$. Then
\begin{align*}
\|T(u)\| &=\lim \frac{\|T(x_n)-T(y_n)\|}{\|x_n-y_n\|}\geq \lim \frac{\bigl|\|T(x_n)\|-\|T(y_n)\|\bigr|}{\|x_n-y_n\|} \\ &= \lim \frac{|p(x_n)-p(y_n)|}{\|x_n-y_n\|}=\|p\|.\qedhere
\end{align*}
\end{proof}

As a consequence of this result, we may provide an interesting example which shows that there is a Lipschitz version of James theorem, but in this case it characterizes finite-dimensionality instead of reflexivity.

\begin{example}\label{example:NAseminorm}
{\slshape For every infinite-dimensional Banach space $X$, there is a continuous seminorm $p\in\mathrm{Sem}(X)$ which does not attain its norm.\ } As a consequence, {\slshape if $\DA(X)=\Lip(X)$ for a Banach space $X$, then $X$ is finite-dimensional.}
\end{example}

The example is based on \cite[Lemma~2.2]{MarMerPay}, but we include the details for the sake of completeness.

\begin{proof}
Since $X$ is infinite-dimensional, there is a sequence $\{x_n^*\}$ in $S_{X^*}$ which is weak-star convergent to $0$ (this is the Josefson-Nissenzweig theorem). Now, consider the seminorm $p:X\longrightarrow \R$ given by
$$
p(x)=\max\left\{\frac{n\,|x_n^*(x)|}{n+1}\,:\, n\in\N\right\} \qquad (x\in X)
$$
and observe that $\|p\|=1$ but $p(x)<1$ for every $x\in S_X$.
\end{proof}

We are ready to prove the uniform density of the set of norm attaining seminorms. Actually, we may prove more.

\begin{prop}[Bishop-Phelps-Bollob\'{a}s theorem for seminorms in the uniform norm]\label{prop:seminDensistyUnifNorm}
Let $X$ be a Banach space. Then for every $\eps>0$ there is such a $\delta>0$ that for every $p_0\in\mathrm{Sem}(X)$ with $\|p_0\|=1$ and every $x_0\in S_X$ with $p_0(x_0)>1-\delta$, there exist $p\in\mathrm{Sem}(X)$ with $\|p\|=1$ and $x\in S_X$ such that
$$
p(x)=1=\|p\|,\qquad \|x-x_0\|<\eps,\quad \text{and} \quad \|p-p_0\|_\infty =\sup_{x\in S_X}|p(x)-p_0(x)| <\eps.
$$
\end{prop}

\begin{proof}
Let $\Gamma$ be a set whose cardinality equals to the density character of $X$. It is shown in \cite[Theorem~2.2]{AAGM2}, that for every $\eps>0$, there is $\delta>0$ such that whenever $T_0\in L(X,\ell_\infty(\Gamma,\R))$ with $\|T_0\|=1$ and $x_0\in S_X$ satisfy that $\|T_0(x_0)\|>1-\delta$, there exist $T\in L(X,\ell_\infty(\Gamma,\R))$ with $\|T\|=1$ and $x\in S_X$ such that
\begin{equation}\label{eq:seminorm}
\|Tx\|=1=\|T\|,\qquad \|x-x_0\|<\eps,\qquad \|T-T_0\|<\eps.
\end{equation}
Now, for a given $\eps>0$, consider such a $\delta>0$.
By Remark~\ref{seminorms-remark}, there is $T_0\in L(X,\ell_\infty(\Gamma,\R))$ such that $p_0(x)=\|T_0(x)\|$ for all $x\in X$. As $p_0(x_0)>1-\delta$, we have $\|T_0(x_0)\|>1-\delta$ and we may find $T\in L(X,\ell_\infty(\Gamma,\R))$ with $\|T\|=1$ and $x\in S_X$ satisfying \eqref{eq:seminorm}. Consider $p\in\mathrm{Sem}(X)$ be given by $p(z)=\|T(z)\|$ for every $z\in X$. Then, $\|p\|=1=p(x)$, $\|x-x_0\|<\eps$ and
$$
|p(z)-p_0(z)|=\bigl|\|T(z)\|-\|T_0(z)\|\bigr| \leq \|T(z)-T_0(z)\|\leq \|T-T_0\|<\eps
$$
for every $z\in S_X$, so $\|p-p_0\|_\infty<\eps$, finishing the proof.
\end{proof}

One may wonder whether the above result is also true replacing $\|p-p_0\|_\infty$ by $\|p-p_0\|_{\textrm{Lip}}$. In general, $\|p-q\|_\infty \leq \|p-q\|_{\textrm{Lip}}$ for all $p,q\in\mathrm{Sem}(X)$, so the uniform convergence of seminorms is weaker than the Lipschitz convergence, and the next example shows that it is indeed strictly weaker.

\begin{remark}
{\slshape The uniform convergence of seminorms does not force the Lipschitz convergence, even in the finite-dimensional case.\ } Indeed, let $X$ be the two dimensional space $\R^2$ endowed with the maximum norm, let $p_0\in\mathrm{Sem}(X)$ be defined by $p_0(x_1,x_2)=|x_1|$ for every $(x_1,x_2)\in X$, and for every $n\in \N$ let $p_n\in\mathrm{Sem}(X)$ be defined by $p_n(x_1,x_2)=\max\{|x_1|,\tfrac{1}{n}|x_2|\}$ for every $(x_1,x_2)\in X$. On the one hand, $\|p_n-p_0\|_\infty\leq 1/n$, so $\{p_n\}$ is uniformly convergent to $p_0$. On the other hand,
\begin{equation*}
\|p_n-p_0\|_{\textrm{Lip}}\geq \frac{\bigl|\bigl(p_n(0,n)-p_0(0,n)\bigr) - \bigl(p_n(1,n)-p_0(1,n)\bigr)\bigr|}{\|(0,n)-(1,n)\|}=1.\qedhere
\end{equation*}
\end{remark}

Our last result in this section shows that the Radon-Nikod\'{y}m property is sufficient to assure that norm attaining seminorms are dense in the Lipschitz sense.

\begin{prop}[Bishop-Phelps theorem for seminorms in the Lipschitz norm for RNP spaces]\label{prop:semiRNP} Let $X$ be a Banach space with the Radon-Nikod\'{y}m property. Then, the set of norm attaining seminorms is Lipschitz-norm dense in $\mathrm{Sem}(X)$.
\end{prop}

\begin{proof}
It is an easy consequence  of the following (quite not easy) Stegall's result \cite[Theorem on page 174]{Stegall} from which we are citing only the part which we need: let $D$  be an RNP set and $ f :D \longrightarrow \R$ be upper semicontinuous and bounded above. Then, for every $\eps >0$, there exists $x^* \in X^*$ with $\|x^*\| < \eps$ such that $f + |x^*|$ attains its supremum on $D$. Now, let us apply this theorem to $D = B_X$ and $f = p \in\mathrm{Sem}(X)$. Then, the corresponding   $q := p + |x^*|$ is again a continuous seminorm, and $q$ is norm-attaining. Finally,
\begin{equation*}
\|p-q\|_{\textrm{Lip}} = \bigl\| |x^*|   \bigr\|_{\textrm{Lip}}=  \bigl\| |x^*|  \bigr\|_\infty = \|x^*\| < \eps.\qedhere
\end{equation*}
\end{proof}


\section{Two preliminary results}\label{sec:freeLipspace}
In this section we demonstrate two preliminary results on the way to Theorem~\ref{main theorem}. The first one is a weak version of the Bishop-Phelps-Bollob\'{a}s theorem for Lipschitz functionals, valid for all Banach spaces, which can be of independent interest.

\begin{lemma}[Preliminary LipBPB Theorem] \label{lemLipBPB-prel}
Let $X$ be a  Banach space, $f\in \Lip (X)$, $\|f\| = 1$, $\delta \in (0, 2)$ and let  $x,y \in X, x\neq y$ be such elements that
\beq\label{eps-att-ineq}
 \frac{f(x)-f(y)}{\|x-y\|}>1-\delta.
\eeq
Then for every $h \in \Lip (X)$ with $\|h\| = 1$ and $ \frac{ h(x)-h(y)}{\|x-y\|} = 1$, there exists $g \in \Lip (X)$ with $\|g\|=1$, $\|f -g\| < \sqrt{2 \delta}$ and there exists a sequence of pairs $\{(v_n,w_n)\}$ in $X\times X$ with $v_n\neq w_n$ for every $n$, such that
$$
\frac{ h(v_n)-h(w_n)}{\|v_n-w_n\|}>1-\sqrt{2 \delta} \ \text{ for all $n \in \N$} \quad \text{and}\quad \lim_{n \to \infty}\frac{g(v_n)-g(w_n)}{\|v_n-w_n\|} = 1.
$$
\end{lemma}

The proof of this result is based on the \textcolor[rgb]{1.00,0.00,0.00}{Lipschitz-free} space technique, so let us first recall the relevant definitions and basic facts. For every $x \in X$, we denote $\hat{x}$ the corresponding evaluation functional on $\Lip(X)$, i.e.\ $\hat{x}(f)=f(x)$. Then $\hat{x}$ is an element of $\Lip(X)^*$. The subspace $\overline{\mathrm{Lin}}\{\hat{x}\,:\, x\in X\}$ of $\Lip(X)^*$ is denoted by $\F$. The most common name for $\F$ is the \emph{Lipschitz-free space} of $X$. This object was studied under various names by several authors  (\cite{Arens-Eells}, \cite{Kadets}, \cite{GodKalt}), and is known to be useful for Lipschitz maps study.

The elements of $\Lip(X)$ are continuous linear functionals on  $\F$, moreover $\F^* = \Lip(X)$ as Banach spaces.  The map $x \longmapsto \hat{x}$ is a non-linear isometric embedding of $X$ into $\F$ since $\|\hat{x}-\hat{y}\|_{\F}=\|x-y\|_X$ for all $x,y\in X$.

The action of $f \in \Lip(X)$ on $w \in \F$ is denoted by $\langle f, w \rangle$. With this notation, $\langle f, \hat{x} \rangle = f(x)$ and the formula \eqref{lipnorm} can be re-written as follows:
\beq \label{formulalipnorm}
\|f\|=\sup\left\{\left|\left\langle f, \frac{\hat{x}-\hat{y}}{\|x-y\|}\right\rangle\right| \,:\, x,y \in X, x\neq y\right\}.
\eeq
Denote $W = \left\{\frac{\hat{x}-\hat{y}}{\|x-y\|}\,:\, x,y \in X, x\neq y\right\}$ and observe that $W$ is a symmetric subset of $S_\F$. The Hahn-Banach theorem, together with formula \eqref{formulalipnorm}, gives us the following result:
\begin{equation}\label{Bconv}
B_{\F}=\cconv\, W.
\end{equation}
We are now ready to prove our result.

\begin{proof}[Proof of Lemma~\ref{lemLipBPB-prel}]
Consider $w = \frac{\hat{x}-\hat{y}}{\|x-y\|} \in S_\F$. The condition \eqref{eps-att-ineq} gives us that $\langle f, w \rangle > 1-\delta$.  Since  $f \in \Lip (X) = \F^*$, the Bishop-Phelps-Bollob\'{a}s theorem (Theorem~\ref{th:BPB}) is applicable. So, there are $g \in \Lip (X)$ with $\|g\|=1$ and $z \in \F$ with $\|z\|=1$ such that $\|w - z\| < \sqrt{2 \delta}$, $\|f -g\| < \sqrt{2 \delta}$, and $\langle g, z \rangle = 1$. Let $\nu > 0$ be such a number that $\|w - z\| < \nu < \sqrt{2 \delta}$.  Fix a sequence $\{\delta_n\}$ of positive numbers converging to $0$. The formula \eqref{Bconv} implies that we can select a sequence $\{z_n\}$ in $\conv W$ converging to $z$ in such a way that
\beq\label{z_n-def-ieq}
\|w - z_n\| < \nu,  \quad  \text{and} \quad  \langle g, z_n \rangle > 1 - \delta_n.
\eeq
The condition on $h$ means $\langle h, w \rangle = 1$, consequently
\beq\label{z_n-def-ieq++}
 \langle h, z_n \rangle \ge \langle h, w \rangle - \|w - z_n\|  > 1 - \nu.
\eeq
Choose a sequence $\{\alpha_n\}$ in $(0, 1)$ satisfying the conditions
\beq\label{alphan-ineq}
\lim_{n \to \infty}\alpha_n = 0  \quad \text{and} \quad \lim_{n \to \infty}  \frac{\delta_n }{ \alpha_n} = 0.
\eeq
Fix $n\in \N$. Combining inequalities \eqref{z_n-def-ieq} and  \eqref{z_n-def-ieq++}, we obtain
$$
\alpha_n  \langle h, z_n \rangle + (1 - \alpha_n) \langle g, z_n \rangle > \alpha_n(1-\nu) + (1-\alpha_n)(1-\delta_n)=
 1 - \alpha_n \nu - (1 - \alpha_n) \delta_n.
$$
Since $z_n \in \conv W$, the last inequality implies that there exists $u_n \in W$ such that
$$
\alpha_n  \langle h, u_n \rangle + (1 - \alpha_n) \langle g, u_n \rangle > 1 - \alpha_n \nu - (1 - \alpha_n) \delta_n.
$$
Combining this fact with the evident estimates $\langle h, u_n \rangle \le 1$ and $\langle g, u_n \rangle \le 1$, we deduce that
\begin{align*}
 \langle g, u_n \rangle &> 1 - \delta_n - \frac{\alpha_n}{1 - \alpha_n}\sqrt{2 \delta}
\intertext{and}
 \langle h, u_n \rangle &> 1 - \nu - \delta_n \frac{1- \alpha_n}{ \alpha_n}.
 \end{align*}
These inequalities, together with \eqref{alphan-ineq}, imply that
$$
\langle g, u_n \rangle \longrightarrow 1
$$
and that
$$
\langle h, u_n \rangle > 1 -  \sqrt{2 \delta}
$$
for $n\in \N$ large enough.
In order to complete the proof, it remains to recall that every $u_n \in W$ is of the form $\frac{\hat{v}_n-\hat{w}_n}{\|v_n-w_n\|}$ for some $v_n,w_n \in X$ with $v_n\neq w_n$.
\end{proof}

Before stating the second result, we need a couple of definitions.

\begin{definition}\label{loc-dir-att def}
A functional $g \in \Lip (X)$ \emph{attains its norm in a point $v \in X$ at the direction} $u\in S_X$ if there is a sequence of pairs $\{(x_n,y_n)\}$ in $X\times X$, with $x_n\neq y_n$, such that
\begin{equation*} 
\lim_{n \to \infty} x_n = \lim_{n \to \infty} y_n = v,  \quad \lim_{n \to \infty}\frac{x_n-y_n}{\|x_n-y_n\|} = u \quad \text{and} \quad  \lim_{n \to \infty}\frac{ g(x_n)-g(y_n)}{\|x_n-y_n\|} = \|g\|.
\end{equation*}
In this case, we say that $g$ \emph{attains its norm locally-directionally}. The set of all those $f \in \Lip (X)$ that attain their norm locally-directionally is denoted by $\LDA(X)$.
\end{definition}

\begin{definition} \label{loc-dir-LipBPB def}
A Banach space $X$ has the \emph{local directional Bishop-Phelps-Bollob\'{a}s property for Lipschitz functionals} ($X \in  \LLipBPB$ for short), if for every $\varepsilon>0$ there is such a $\delta>0$, that for every $f\in \Lip (X)$ with  $\|f\| = 1$ and for every $ x,y \in X$ with $x\neq y$ satisfying $ \frac{ f(x)-f(y)}{\|x-y\|}>1-\delta$, there is $g \in \Lip (X)$ with $\|g\|=1$ and there are $v \in X$, $u \in S_X$ such that $g$ attains its norm in the point $v$ at the direction $u$,
$\|g-f\|<\varepsilon$,   $\left\|\frac{x-y}{\|x-y\|}-u\right\|<\varepsilon$, and $\dist(v,  \conv\{x, y\}) < \eps$.
\end{definition}

The second preliminary result of this section, which also can be of independent interest, is a relaxation of the requirements for a Banach space to have the $ \LLipBPB$.

\begin{lemma} \label{lemma-relaxed}
Let $X$ be a Banach space. Suppose that for every $\varepsilon>0$ there is such a $\delta>0$ that for every $f\in \Lip (X)$ with  $\|f\| = 1$ and for every pair $(x,y) \in X\times X$, $x\neq y$ with $\frac{ f(x)-f(y)}{\|x-y\|}>1-\delta$ there is $g \in \Lip (X)$ with $\|g\|=1$ and a sequence of pairs $\{(v_n,w_n)\}$ in $X\times X$ with $v_n\neq w_n$ for every $n$, such that
\beq \label{ineq-oj}
\lim_{n \to \infty}\frac{g(v_n)-g(w_n)}{\|v_n-w_n\|} = 1,
\eeq
$\|g-f\|<\varepsilon$,  $\left\|\frac{x-y}{\|x-y\|}-\frac{v_n-w_n}{\|v_n-w_n\|}\right\| <\varepsilon$, $\|v_n-w_n\| < \eps$, and $\dist(v_n,  \conv\{x, y\}) < \eps$. Then $X \in \LLipBPB$.
\end{lemma}

Observe that the difference between the requirements of the lemma above and the local directional Bishop-Phelps-Bollob\'{a}s property is that here the convergence of the sequences $\left\{ v_n \right\}$,  $\left\{ w_n \right\}$ and $\left\{\frac{v_n-w_n}{\|v_n-w_n\|}\right\}$ is not required since these sequences depend upon $\varepsilon$ (but they are still well-controlled).

\begin{proof}
For a fixed $\varepsilon>0$ let us select a decreasing sequence $\{\varepsilon_n\}$ of positive number such that $\sum_{n=1}^\infty \varepsilon_n < \varepsilon/4$. For every $n \in \N$, let $\delta_n=\delta(\varepsilon_n)$ be from the assumptions of the lemma for $\varepsilon_n$. We will demonstrate that $\delta = \delta_1$ satisfies conditions of Definition~\ref{loc-dir-LipBPB def} for $\varepsilon$.

To this end, let us fix $f\in \Lip(X)$ with $\|f\|=1$ and a pair $(x,y)\in X\times X$, $x\neq y$, such that $\frac{f(x)-f(y)}{\|x-y\|}>1-\delta$. Applying the hypotheses to $\varepsilon_1=\varepsilon$, $\delta_1=\delta(\varepsilon_1)$, the norm-one Lipschitz functional $f_1 = f$ and the pair $(x_1,y_1)=(x,y)$ in $X\times X$ which satisfy $\frac{ f_1(x_1)-f_1(y_1)}{\|x_1-y_1\|}>1-\delta_1$, we get the corresponding $g_1$ and the sequence of pairs $\{(v_n, w_n)\}$. Thanks to \eqref{ineq-oj}, we can find such an $n_1\in \N$ that
$$
\frac{g_1(v_{n_1})-g_1(w_{n_1})}{\|v_{n_1}-w_{n_1}\|} > 1-\delta_2.
$$
Let us denote $f_2 = g_1$, $x_2 = v_{n_1}$, and  $y_2 = w_{n_1}$. Then, $\|f_1 - f_2\| < \varepsilon_1$, $\left\|\frac{x_1-y_1}{\|x_1-y_1\|} -\frac{x_2-y_2}{\|x_2-y_2\|}\right\|<\varepsilon_1$, $\dist(x_2,  \conv\{x_1, y_1\}) < \eps_1$,  $\|x_2 -  y_2\| < \eps_1$ and
$$
\frac{ f_2(x_2)-f_2(y_2)}{\|x_2-y_2\|}>1-\delta_2.
$$
The last condition enables us to apply again the hypotheses of the lemma to $\varepsilon_2$, $\delta_2$, $f_2$ and $(x_2,y_2)$ in order to get the corresponding  $f_3$ and $(x_3,y_3)$. Repeating this process, we obtain sequences $\{f_n\}$ in $\Lip(X)$ with $\|f_n\| = 1$ and $\{(x_n, y_n)\}$ in $X\times X$ with $x_n\neq y_n$ having  the following properties:
\begin{enumerate}
\item[(a)] $\|f_n - f_{n+1}\| < \varepsilon_n$,
\item[(b)] $\left\|\frac{x_{n}-y_{n}}{\|x_{n}-y_{n}\|} -\frac{x_{n+1}-y_{n+1}}{\|x_{n+1}-y_{n+1}\|}\right\|<\varepsilon_n$,
\item[(c)] $\frac{ f_n(x_n)-f_n(y_n)}{\|x_n-y_n\|}>1-\delta_n$,
\item[(d)] $\dist(x_{n+1},  \conv\{x_n, y_n\}) < \eps_n$,
\item[(e)] $\|x_{n} -  y_{n}\| < \eps_{n-1}$.
\end{enumerate}
Conditions (d) and (e) imply that $\|x_{n+1} - x_n\| < 2\eps_{n-1}$, $n = 2,3,\ldots$. Consequently, the sequence $\{x_n\}$ has some limit $v \in X$, and by (e) the sequence $\{y_n\}$ has the same limit $v$. For this $v$ we have
$$
\dist(v,  \conv\{x, y\})  \le \dist(x_2,  \conv\{x_1, y_1\}) + \|x_2 - v\| < \eps_1 + \sum_{n=2}^\infty\|x_n - x_{n+1}\| < \eps.
$$
The condition (a)  implies that the sequence $\{f_n\}$ has some limit $g \in \Lip(X)$ with $\|g\|=1$, and the condition (b)  implies  that the sequence $\left\{\frac{x_n-y_n}{\|x_n-y_n\|}\right\}$ has some limit $u \in S_X$.
Moreover,
$$
\|f - g\| < \sum_{n=1}^\infty \varepsilon_n < \varepsilon \quad \text{and} \quad \left\| \frac{x_1-y_1}{\|x_1-y_1\|}-u \right\|<\sum_{n=1}^\infty \varepsilon_n  <\varepsilon.
$$
Also,
$$
 \frac{ g(x_n)- g(y_n)}{\|x_n-y_n\|} \ge 1-\delta_n - \|g - f_n\|,
$$
so $\lim_{n\to\infty} \frac{ g(x_n)- g(y_n)}{\|x_n-y_n\|} = 1$, which proves that $g$ attains its norm at $v$ in the direction $u$.
\end{proof}


\section{Bishop-Phelps-Bollob\'{a}s theorem for uniformly convex spaces}\label{sec:main}

Let us recall the well-known concept of uniform convexity.

\begin{definition} \label{def-unif-conv}
A Banach space $X$ is said to be \emph{uniformly convex}, if for every $\eps>0$ there is such a $\delta>0$, that for every pair $x,y\in B_X$  the condition $\|x-y\|\geq \eps$ implies $\left\|\frac{x+y}{2}\right\|\leq 1-\delta$. (Equivalently,  $\|(x+y)/2\|>1-\delta \Rightarrow \|x-y\|< \eps$). The best possible value of $\delta$ is denoted $\delta_X(\eps)$.
\end{definition}

The unit ball of a uniformly convex space has many small slices. Recall, that if $X$ is a Banach space, for given  $x^*\in S_{X^*}$ and $\delta > 0$ the corresponding \emph{slice} of the unit ball is defined as  $S(B_X,x^*,\delta):=\{x\in B_X \,:\, x^*(x)> 1-\delta\}$. The following easy result states a ``uniform way'' to find small slices on a uniformly convex space. A proof of it can be found in \cite[Lemma~2.1]{ABGM}.

\begin{lemma}\label{lemma-unif-conv-slice}
Let $X$ be a uniformly convex space and $\eps>0$. Then
$$
\diam\,S\bigl(B_X,f,\delta_X(\eps)\bigr)< \eps
$$
for every  $f\in S_{X^*}$ and every $\eps>0$.
\end{lemma}

We may now state and prove the main result of the paper.

\begin{theorem}\label{main theorem}
Every uniformly convex Banach space $X$ has the local directional Bishop-Phelps-Bollob\'{a}s property for Lipschitz functionals.
\end{theorem}

This result implies, in particular, that {\slshape for a uniformly convex Banach space $X$, the set of those Lipschitz functions which attain their norm (locally-)directionally is dense in the space $\Lip(X)$ (with the Lipschitz norm)}. We do not know whether this density holds in every Banach space $X$.

\begin{proof}[Proof of Theorem~\ref{main theorem}]
For a fixed $\varepsilon \in (0, 1/2)$ let us chose such a $\delta \in (0, \eps^2/2)$ that $\sqrt{2 \delta} < \frac12 \delta_X(\varepsilon)$.  Let $f\in \Lip (X)$ with   $\|f\| = 1$ and $ x,y \in X$, $x\neq y$, such that $ \frac{f(x)-f(y)}{\|x-y\|}>1-\delta$. Select $\tilde x, \tilde y \in \conv\{x,y\}$ in such a way that $\|\tilde x - \tilde y\| < \frac14 \min\{\eps, \|\tilde x\|, \| \tilde y\|\}$, the vector  $\tilde x\tilde y$ looks at the same direction that $xy$ (i.e. $\frac{ \tilde x - \tilde y}{\|\tilde x - \tilde y\|} = \frac{ x-y}{\|x-y\|}$) and such that still $ \frac{f(\tilde x)-f(\tilde y)}{\|\tilde x-\tilde y\|}>1-\delta$.

Define $F \in \Lip(X)$ by the formula $F(z) = \max\{\|\tilde x - \tilde y\| - \| \tilde x - z\| , 0\}$. Then $\|F\| = 1$ and $ \frac{F(\tilde x)-F(\tilde y)}{\|\tilde x-\tilde y\|} = 1$.
 Let us denote $x^* \in S_{X^*}$ the supporting functional at the point $ \frac{\tilde  x-\tilde y}{\|\tilde x - \tilde y\|}$. Then, by linearity,  $ \frac{ x^*(\tilde x)-x^*(\tilde y)}{\|\tilde x - \tilde y\|} =1$, so we can apply the preliminary LipBPB Theorem (Lemma~\ref{lemLipBPB-prel}) with $f$, $(x,y)$ and $h=\frac12(F +x^*)\in \Lip(X)$. According to it,
there exist $g \in \Lip (X)$ with $\|g\|=1$, $\|f -g\| < \sqrt{2 \delta} < \eps$ and a sequence of pairs $\{(v_n,w_n)\}$ in $X\times X$ with $v_n\neq w_n$, such that
\beq \label{eqfrachx_n}
h\left(\frac{ v_n-w_n}{\|v_n-w_n\|}\right) = \frac12 \left(\frac{ F(v_n)-F(w_n)}{\|v_n-w_n\|} + \frac{ x^*(v_n)-x^*(w_n)}{\|v_n-w_n\|}\right) > 1-\sqrt{2 \delta}
\eeq
for all $n \in \N$, and
$$
\lim_{n \to \infty}\frac{g(v_n)-g(w_n)}{\|v_n-w_n\|} = 1.
$$
The inequality \eqref{eqfrachx_n} and the fact that $ \|x^*\| = \|F\| = 1$ imply that
\beq \label{eqfrachx_nn}
\frac{ F(v_n) - F(w_n)}{\|v_n-w_n\|}> 1 -  2\sqrt{2\delta}>1-2\eps,
\eeq
and
$$
\frac{ x^*(v_n)-x^*(w_n)}{\|v_n-w_n\|} > 1 - 2\sqrt{2\delta} > 1 - \delta_X(\varepsilon).
$$
The last condition means geometrically that $\frac{ v_n-w_n}{\|v_n-w_n\|} \in S(B_X,x^*, \delta_X(\varepsilon))$. Since also $ \frac{ x-y}{\|x-y\|} = \frac{\tilde  x-\tilde y}{\|\tilde x - \tilde y\|} \in S(B_X, x^*, \delta_X(\varepsilon))$, we get from Lemma~\ref{lemma-unif-conv-slice} that
$$
\left\|\frac{x-y}{\|x-y\|}-\frac{v_n-w_n}{\|v_n-w_n\|}\right\| <\varepsilon
$$
for every $n\in \N$. The function $F$ takes only non-negative values, and the condition \eqref{eqfrachx_nn} implies that $ F(v_n) - F(w_n) > 0$, so $v_n \in \supp F$. Since the maximal possible value of $F$ is $\|\tilde x - \tilde y\| < \eps/4$,  the condition \eqref{eqfrachx_nn} implies also that
$$
\|v_n-w_n\|  < \frac{F(v_n) - F(w_n)}{1-2\eps} < \frac{\eps}{4(1-2\eps)} < \eps.
$$
As $v_n \in \supp F$, we have that $\|v_n - \tilde x\| < \|\tilde x - \tilde y\| < \eps$, so $\dist(v_n,  \conv\{x, y\}) < \eps$.

We have verified all the conditions of  Lemma~\ref{lemma-relaxed}. The application of that Lemma shows that $X$ has the local directional Bishop-Phelps-Bollob\'{a}s property for Lipschitz functionals.
\end{proof}

\vspace{1ex}

{\bf Acknowledgement.} First author partially supported by Spanish grants:  Junta de Andaluc\'{\i}a and FEDER grant FQM-185, MINECO grant MTM2014-57838-C2-1-P and Fundaci\'on S\'eneca, Regi\'on de Murcia  grant 19368/PI/14. Second author partially supported by Spanish MINECO and FEDER project no.\ MTM2012-31755 and by Junta de Andaluc\'{\i}a and FEDER grant FQM-185. Third author partially supported by a grant from Akhiezer's Fund, 2015.

\bibliographystyle{amsplain}

\end{document}